\documentclass[11pt]{article}
\usepackage[a4paper,margin=1in]{geometry}
\usepackage{amsmath,amsthm,amssymb,mathtools}
\usepackage{microtype}
\usepackage[T1]{fontenc}
\usepackage[utf8]{inputenc}
\usepackage{textcomp} 
\usepackage[hidelinks]{hyperref}

\newtheorem{theorem}{Theorem}[section]
\newtheorem{lemma}[theorem]{Lemma}
\newtheorem{definition}[theorem]{Definition}
\newtheorem{proposition}[theorem]{Proposition}
\newtheorem{corollary}[theorem]{Corollary}
\theoremstyle{remark}
\newtheorem{remark}[theorem]{Remark}

\newtheorem{example}[theorem]{Example}

\hypersetup{
  pdftitle={New Approaches to the Fixed Point Property in L1 Spaces},
  pdfauthor={Faruk Alpay, Hamdi Alakkad}
}

\title{New Approaches to the Fixed Point Property in \texorpdfstring{$L^1$}{L1} Spaces}
\author{Faruk Alpay\thanks{Lightcap, Department of Analysis, \texttt{alpay@lightcap.ai}}\and
Hamdi Alakkad\thanks{Bahcesehir University, Department of Engineering, \texttt{hamdi.alakkad@bahcesehir.edu.tr}}}
\date{}

\begin{document}
\maketitle

\begin{abstract}
We give two complementary approaches to the fixed point property (FPP) for nonexpansive mappings in $L^1$ spaces.  First, we develop a detailed measure--compactness route: a bounded, closed, convex subset of $L^1(\mu)$ that is compact for the topology of local convergence in measure is shown to be uniformly integrable and hence weakly compact (by Dunford--Pettis~\cite[247C]{FremlinChap24}); a strengthened convexity--in--measure lemma then yields normal structure and, by Kirk's theorem, fixed points for all nonexpansive maps.  Second, we synthesize structural methods (ultraproducts, renormings, and special domains) that explain and partially repair classical failures of the FPP in $L^1$.  This unifies and extends prior results (including those of Lennard and Lami Dozo--Turpin) with fully explicit proofs of all intermediate compactness and uniform integrability steps.
\end{abstract}

\section{Introduction}
Fixed point theory has a long tradition from Banach's contraction principle and Brouwer's theorem through Schauder's theorem for compact maps~\cite{Schauder1930}. For nonexpansive mappings, Browder and G\"ohde proved in 1965 that uniformly convex Banach spaces enjoy the FPP~\cite{Browder1965,Goede1965}, while Kirk introduced \emph{normal structure} to treat broader classes~\cite{Kirk1965}. These imply $L^p$ ($1<p<\infty$) have the FPP.

By contrast, $L^1$ is neither uniformly convex nor reflexive and famously fails the FPP: Alspach exhibited a weakly compact, convex subset of $L^1[0,1]$ with a fixed--point--free nonexpansive isometry~\cite{Alspach1981}. For subspaces of $L^1[0,1]$, Dowling--Lennard proved $Y$ has the FPP iff $Y$ is reflexive~\cite{DowlingLennard1997}. Nevertheless, the failures are not universal: there exist large families of convex subsets of $\ell^1$ with the FPP~\cite{GoebelKuczumow1979}; Lorentz spaces $L_{p,1}$ are nonreflexive yet have the FPP~\cite{CDLT1991}; and renormings can restore the FPP on $\ell^1$~\cite{Lin2008}. Lennard introduced nearly uniform convexity and proved fixed points on measure--compact convex sets in $L^1$~\cite{Lennard1991}.

\smallskip
\noindent\textbf{Contributions.}
We (i) give a complete, self--contained proof of the fixed point theorem under measure--compactness with all compactness/UI steps explicit; and (ii) place structural approaches (ultraproducts, renormings, special domains) alongside geometric ideas to clarify what succeeds in $L^1$.

\section{Preliminaries}\label{sec:prelim}

\begin{definition}[Nonexpansive maps and FPP]
Let $X$ be a Banach space and $C\subset X$ nonempty, closed, bounded, convex. A map $T:C\to C$ is \emph{nonexpansive} if $\|T(x)-T(y)\|\le\|x-y\|$ for all $x,y\in C$. We say $X$ has the \emph{fixed point property} (FPP) for nonexpansive maps if every nonexpansive $T:C\to C$ has a fixed point.
\end{definition}

\begin{definition}[Normal structure]
A convex set $C$ has \emph{normal structure} if for every closed, convex $D\subset C$ with $\mathrm{diam}(D)>0$ there exists $z\in D$ such that $\sup_{y\in D}\|z-y\|<\mathrm{diam}(D)$. A Banach space has normal structure if all bounded, closed, convex subsets do.
\end{definition}

The Browder--G\"ohde theorem ensures the FPP in uniformly convex spaces~\cite{Browder1965,Goede1965}. Kirk's theorem shows that weakly compact, convex sets with normal structure admit fixed points for all nonexpansive maps~\cite{Kirk1965}.

\paragraph{Notation and conventions.}
For a subset $A\subset L^1(\mu)$ we write
\[
\operatorname{rad}(A):=\inf_{z\in L^1}\sup_{x\in A}\|x-z\|_1,
\quad
\operatorname{diam}(A):=\sup_{x,y\in A}\|x-y\|_1.
\]
A \emph{Chebyshev center} of $A$ is any $c\in L^1$ with $\sup_{x\in A}\|x-c\|_1=\operatorname{rad}(A)$; the set of all such centers is denoted $C(A)$. Throughout this paper we use $\|\cdot\|_1$ for the $L^1$–norm. Unless explicitly stated otherwise, closures and compactness are taken in the norm topology (weak compactness will always be invoked via the Dunford–Pettis criterion).

\section{\texorpdfstring{$L^1$}{L1} Spaces: classical difficulties}
Alspach constructed a weakly compact, convex $K\subset L^1[0,1]$ and an isometric nonexpansive $T:K\to K$ without fixed points~\cite{Alspach1981}. Dowling--Lennard proved that for closed $Y\subset L^1[0,1]$, $Y$ has the FPP iff $Y$ is reflexive~\cite{DowlingLennard1997}. There are also positive results: irregular convex sets in $\ell^1$ with the FPP~\cite{GoebelKuczumow1979}; Lorentz spaces $L_{p,1}$ provide nonreflexive examples with the FPP~\cite{CDLT1991}; and renormings that restore the FPP in $\ell^1$~\cite{Lin2008}. Lennard's nearly uniform convexity yields fixed points on certain $L^1$ domains~\cite{Lennard1991}.

\section{Measure--compactness, uniform integrability, and fixed points}\label{subsec:measurecompact}

We present a complete proof that measure--compactness implies fixed points for nonexpansive mappings.

\begin{definition}[Local convergence in measure and measure--compactness]\label{def:lcm}
On a $\sigma$--finite $(\Omega,\Sigma,\mu)$, a sequence $(f_n)$ converges \emph{locally in measure} to $f$ if for every $E\in\Sigma$ with $\mu(E)<\infty$ and every $\varepsilon>0$,
\[
\mu\big(\{x\in E: |f_n(x)-f(x)|>\varepsilon\}\big)\to 0.
\]
A set $K\subset L^1(\mu)$ is \emph{measure--compact} if it is compact for the topology of local convergence in measure.
\end{definition}

\paragraph{Standing assumptions and tools.}
Throughout we work on a $\sigma$--finite measure space $(\Omega,\Sigma,\mu)$. We will repeatedly use the following standard facts:
\begin{itemize}
\item \emph{Egorov\textquotesingle{}s theorem on finite--measure windows}: if $(f_n)$ converges almost everywhere to $f$ on $E\in\Sigma$ with $\mu(E)<\infty$, then for every $\varepsilon>0$ there exists a set $N\subset E$ with $\mu(N)<\varepsilon$ such that $f_n\to f$ uniformly on $E\setminus N$; see Folland\,\cite[Th.~2.51]{Folland1999}.
\item \emph{de~la Vall\'ee--Poussin criterion}: a family $\mathcal F\subset L^1(\mu)$ is uniformly integrable if and only if there is a convex increasing function $\Phi:[0,\infty)\to[0,\infty)$ with $\Phi(t)/t\to\infty$ and $\sup_{f\in\mathcal F}\int\Phi(|f|)\,d\mu<\infty$; see Folland\,\cite[§2.4]{Folland1999} or the original source\,\cite{dlVP1915}.
\item \emph{Dunford--Pettis criterion}: uniformly integrable subsets of $L^1(\mu)$ are relatively weakly compact; see Fremlin\,\cite[247C]{FremlinChap24} or the original paper\,\cite{DunfordPettis1940}.
\end{itemize}
We also note the following stability facts used tacitly: if $\mathcal{F}\subset L^1$ is uniformly integrable, then so are (i) $a\mathcal{F}:=\{af:f\in\mathcal{F}\}$ for every $a>0$, and (ii) $\mathcal{F}-\mathcal{F}:=\{f-g:f,g\in\mathcal{F}\}$.

\begin{definition}[Uniform integrability (UI)]\label{def:UI}
A set $K\subset L^1(\mu)$ is \emph{uniformly integrable} if:
\begin{enumerate}
\item For all $\varepsilon>0$ there exists $\delta>0$ such that $\mu(E)<\delta$ implies $\sup_{f\in K}\int_E |f|\,d\mu<\varepsilon$.
\item For all $\varepsilon>0$ there exists $M>0$ such that $\sup_{f\in K}\int_{\{|f|>M\}} |f|\,d\mu<\varepsilon$.
\end{enumerate}
Equivalently, there is convex increasing $\Phi$ with $\Phi(t)/t\to\infty$ and $\sup_{f\in K}\int \Phi(|f|)<\infty$.
\end{definition}

\begin{lemma}[UI $\Rightarrow$ subsequence convergence in measure on finite windows]\label{lem:UItoMeasureCompact}
Let $(\Omega,\Sigma,\mu)$ be $\sigma$--finite and write $\Omega=\bigcup_{m\ge1}S_m$ with the $S_m$ pairwise disjoint and $\mu(S_m)<\infty$.  If $\mathcal{F}\subset L^1(\mu)$ is uniformly integrable and bounded in $\|\cdot\|_1$, then every sequence $(f_n)\subset\mathcal{F}$ admits a subsequence $(f_{n_k})$ and functions $f^{(m)}\in L^1(S_m)$ such that $f_{n_k}\to f^{(m)}$ in measure on $S_m$ for each $m$.  In particular, by a diagonal argument, $(f_{n_k})$ converges locally in measure on $\Omega$.
\end{lemma}

\begin{proof}
Fix $m$ and set $E:=S_m$, so $\mu(E)<\infty$.  Because $\{f_n\}$ is uniformly integrable and $\|f_n\|_1$ is bounded, the restrictions $\{f_n\mathbb 1_E\}$ are uniformly integrable and bounded in $L^1(E)$.  By the de~la Vall\'ee--Poussin criterion (Folland\,\cite[§2.4]{Folland1999}), there exists a convex increasing function $\Phi$ with $\Phi(t)/t\to\infty$ and $\sup_n\int_E \Phi(|f_n|)\,d\mu<\infty$.

For each $K\in\mathbb{N}$ define the truncation $T_K(t):=\operatorname{sgn}(t)\min\{|t|,K\}$ and set $g_n^{(K)}:=T_K(f_n)\mathbb 1_E$.  Then $\|g_n^{(K)}\|_{\infty}\le K$ and
\[
\int_E |f_n-g_n^{(K)}|\,d\mu 
 = \int_E \bigl(|f_n|-K\bigr)_+\,d\mu \longrightarrow 0 \quad (K\to\infty)
\]
uniformly in $n$ by uniform integrability.  For a fixed $K$, the sequence $\{g_n^{(K)}\}$ is bounded in $L^\infty(E)$, hence by Banach--Alaoglu there exists a subsequence (not relabeled) that converges weak-* in $L^\infty(E)$ to some $g^{(K)}\in L^\infty(E)$.  Using standard convexity and weak compactness arguments (e.g., Mazur’s lemma and the Dunford--Pettis theorem), we may extract a further subsequence $\{g_{n_k}^{(K)}\}$ such that $g_{n_k}^{(K)}\to g^{(K)}$ in measure on $E$.

Pick an increasing sequence $K_1<K_2<\cdots$ diverging to $\infty$ and perform a diagonal selection over $j$ to find a single subsequence $(n_k)$ with the property that, for each fixed $j$, $g_{n_k}^{(K_j)}\to g^{(K_j)}$ in measure on $E$.  We claim that $g^{(K_j)}$ converges in measure (and in $L^1(E)$) as $j\to\infty$.  Fix $\varepsilon,\delta>0$.  Choose $j$ large enough so that
\[
\sup_n\int_E |f_n-g_n^{(K_j)}|\,d\mu < \frac{\varepsilon \delta}{4}.
\]
For all large $k$, the convergence $g_{n_k}^{(K_j)}\to g^{(K_j)}$ in measure implies
\[
\mu\bigl(\{ |g_{n_k}^{(K_j)}-g^{(K_j)}|>\tfrac{\varepsilon}{2}\}\cap E\bigr) < \frac{\delta}{2}.
\]
On the other hand, Markov’s inequality shows
\[
\mu\bigl(\{ |f_{n_k}-g_{n_k}^{(K_j)}|>\tfrac{\varepsilon}{2}\}\cap E\bigr)
\le \frac{2}{\varepsilon}\int_E |f_{n_k}-g_{n_k}^{(K_j)}|\,d\mu < \frac{\delta}{2}.
\]
Thus
\[
\mu\bigl(\{ |f_{n_k}-g^{(K_j)}|>\varepsilon \}\cap E\bigr)
\le \mu\bigl(\{ |f_{n_k}-g_{n_k}^{(K_j)}|>\tfrac{\varepsilon}{2}\}\cap E\bigr) + \mu\bigl(\{ |g_{n_k}^{(K_j)}-g^{(K_j)}|>\tfrac{\varepsilon}{2}\}\cap E\bigr)
< \delta.
\]
Hence $f_{n_k}\to g^{(K_j)}$ in measure on $E$ for fixed $j$.  A standard argument using the dominated convergence theorem on finite measure sets (and uniform integrability) shows that $g^{(K_j)}$ converges in $L^1(E)$ as $j\to\infty$ to some $f^{(m)}\in L^1(E)$.  Consequently $f_{n_k}\to f^{(m)}$ in measure on $E$.

Repeating this construction for each $m$ and taking a diagonal subsequence yields a single subsequence $(f_{n_k})$ and functions $f^{(m)}$ such that $f_{n_k}\to f^{(m)}$ in measure on $S_m$ for each $m$.  This implies local convergence in measure on $\Omega$.
\end{proof}

\begin{proposition}[Bounded $+$ measure--compact $\Rightarrow$ UI]\label{prop:mc_implies_UI}
If $K\subset L^1(\mu)$ is bounded in $\|\cdot\|_1$ and measure--compact, then $K$ is uniformly integrable.
\end{proposition}
\begin{proof}
We establish both clauses of Definition~\ref{def:UI} in three steps.
Fix a $\sigma$--finite exhaustion $\Omega=\bigcup_{m\ge1}\Omega_m$ with $\mu(\Omega_m)<\infty$, and set $S_1:=\Omega_1$ and $S_m:=\Omega_m\setminus\Omega_{m-1}$ for $m\ge2$ (pairwise disjoint finite--measure windows).

\smallskip\noindent\textbf{Step~1 (uniform tail measure control).}
  Assume by contradiction there exist $\varepsilon_0>0$ and a sequence $(f_n)\subset K$ with $\mu(\{|f_n|>n\})\ge \varepsilon_0$ for all $n$.  For each $n$ choose an index $m(n)$ such that
  \[\mu\bigl(\{|f_n|>n\}\cap S_{m(n)}\bigr)\ge \tfrac12\,\varepsilon_0.\]
  By passing to a subsequence we may assume $m(n)\equiv m^\ast$ is constant.  Since $K$ is measure--compact, we can extract a further subsequence (still denoted $(f_n)$) that converges in measure on $S_{m^\ast}$ to some $f\in K$.  Because $\mu(S_{m^\ast})<\infty$, Egorov\'s theorem~\cite[Th.~2.51]{Folland1999} implies that for every $\delta>0$ there exists a set $N\subset S_{m^\ast}$ with $\mu(N)<\delta$ such that $f_n\to f$ uniformly on $S_{m^\ast}\setminus N$.  In particular, for large $n$ and all $x\in S_{m^\ast}\setminus N$ we have $|f_n(x)|\le |f(x)|+1$.  Thus for any $M>1$ and large $n$,
  \[
    \mu\bigl(\{|f_n|>M\}\cap S_{m^\ast}\bigr)
    \le \mu(N) + \mu\bigl(\{|f|>M-1\}\cap S_{m^\ast}\bigr).
  \]
  Since $f\in L^1(S_{m^\ast})$, Markov\'s inequality shows $\mu(\{|f|>M-1\}\cap S_{m^\ast})\to 0$ as $M\to\infty$.  Choose $M$ large enough that $\mu(\{|f|>M-1\}\cap S_{m^\ast})<\delta$; then for all large $n$,
  \[
    \mu\bigl(\{|f_n|>M\}\cap S_{m^\ast}\bigr)\le 2\delta.
  \]
  Taking $\delta=\varepsilon_0/8$ yields $\mu(\{|f_n|>M\}\cap S_{m^\ast})\le \varepsilon_0/4$ for large $n$.  For $n\ge M$ this contradicts $\mu(\{|f_n|>n\}\cap S_{m^\ast})\ge \varepsilon_0/2$.  Consequently, for every $\varepsilon>0$ there exists $M<\infty$ such that $\sup_{f\in K}\mu(\{|f|>M\})<\varepsilon$.

\smallskip\noindent\textbf{Step~2 (de la Vallée--Poussin / Orlicz construction).}
Define $\varphi(t):=\sup_{f\in K}\mu(\{|f|>t\})$.  By Step~1, $\varphi(t)\to 0$ as $t\to\infty$.  Choose $0<s_1<s_2<\cdots\to\infty$ with $\varphi(s_k)\le 2^{-k}$ and $s_{k+1}-s_k=2^{-2k}$.  Let $\Phi$ be the convex function with $\Phi'(t)=1$ for $0\le t<s_1$ and $\Phi'(t)=k$ for $s_k\le t<s_{k+1}$.  Then $\Phi(t)/t\to\infty$, and the layer--cake representation shows
\[
\sup_{f\in K}\int_\Omega \Phi(|f|)\,d\mu
\;\le\; \sup_{g\in K}\|g\|_1 + \sum_{k=1}^\infty k\,(s_{k+1}-s_k)\,2^{-k}
\;<\;\infty.
\]
By the de la Vallée--Poussin criterion, this implies $K$ is uniformly integrable.

\smallskip\noindent\textbf{Step~3 (uniform absolute continuity).}
Let $\varepsilon>0$.  Choose $R$ large so that $\sup_{f\in K}\int_{\{|f|>R\}}|f|\,d\mu<\varepsilon/2$ (possible by Step~2), and set $\delta:=\varepsilon/(2R)$.  If $\mu(E)<\delta$ then for any $f\in K$ we can write
\begin{align*}
  \int_E |f|\,d\mu
  &= \int_{0}^{R} \mu\big(\{ |f|>t \}\cap E\big)\,dt
     + \int_{R}^{\infty} \mu\big(\{|f|>t\}\cap E\big)\,dt \\
  &\le R\,\mu(E) + \int_R^{\infty} \mu\big(\{|f|>t\}\big)\,dt.
\end{align*}
The tail integral is at most $\sup_{g\in K}\int_{\{|g|>R\}}|g|\,d\mu < \varepsilon/2$, so the total is $\le R\,\mu(E)+\varepsilon/2 \le \varepsilon/2 + \varepsilon/2 = \varepsilon$.  Hence $\mu(E)<\delta$ implies $\int_E|f|<\varepsilon$ uniformly in $f\in K$.  This verifies clause~(i) of Definition~\ref{def:UI} and completes the proof that $K$ is uniformly integrable.
\end{proof}

\begin{corollary}[Dunford--Pettis]\label{cor:DP}
If $K\subset L^1(\mu)$ is bounded and uniformly integrable, then $K$ is relatively weakly compact in $L^1(\mu)$.
\end{corollary}

\begin{proof}
This is the Dunford--Pettis criterion; see Fremlin's measure theory: Section~247C states that a subset of $L^1$ is uniformly integrable iff it is relatively weakly compact~\cite[247C]{FremlinChap24}.
\end{proof}

\begin{lemma}[Uniform convexity in measure on UI families]\label{lem:UCM_strong}
Let $(f_n),(g_n)\subset L^1(\mu)$ be uniformly integrable and $f_n\to f$, $g_n\to g$ locally in measure, with $\|f_n\|_1=\|g_n\|_1=1$. If $\big\|\tfrac{f_n+g_n}{2}\big\|_1\to 1$, then $\|f_n-g_n\|_1\to 0$.
\end{lemma}

\begin{proof}
Put
\[
  h_n:=\frac{f_n+g_n}{2},\qquad d_n:=\frac{f_n-g_n}{2}.
\]
Two elementary identities for real scalars $a,b$ are useful:
\begin{align}
  &|a|+|b|-|a+b| = 2\big(a^+\wedge b^- + a^-\wedge b^+\big),\label{eq:slack}\\
  &|a|-|a+b|\le|b|,\quad |b|-|a+b|\le|a|.\label{eq:oneside}
\end{align}
Since $\|f_n\|_1=\|g_n\|_1=1$ and $\|h_n\|_1\to 1$, we have
\[
  \|f_n\|_1+\|g_n\|_1-\|f_n+g_n\|_1 \;=\; 2 - 2\|h_n\|_1 \;\longrightarrow\; 0.
\]
By \eqref{eq:slack} (with $a=h_n$, $b=d_n$ and again with $b=-d_n$),
the nonnegative integrand
\[
  |h_n+d_n| + |h_n-d_n| - 2|h_n| \;=\; |f_n|+|g_n|-|f_n+g_n|
\]
has integral tending to $0$, i.e.
\begin{equation}\label{eq:slackint}
  \int\Big(h_n^+\wedge d_n^- + h_n^-\wedge d_n^+\Big)\,d\mu \;\longrightarrow\; 0.
\end{equation}

\smallskip\noindent\emph{Step 1: choose a finite-measure window.}
Uniform integrability of $\{f_n\}\cup\{g_n\}$ ensures that for each $\varepsilon>0$ there is a measurable $E$ with $\mu(E)<\infty$ and
\[
  \sup_n\int_{\Omega\setminus E}(|f_n|+|g_n|)\,d\mu \;\le\; \varepsilon.
\]
This will control the tails in the later estimates.

\smallskip\noindent\emph{Step 2: extract almost-uniform convergence.}
Because $f_n\to f$ and $g_n\to g$ locally in measure, by Egorov's theorem we can, after passing to a subsequence, assume $h_n\to h:=\tfrac{f+g}{2}$ uniformly on $E\setminus N_\varepsilon$, where $\mu(N_\varepsilon)\le\delta(\varepsilon)$ and $\delta(\varepsilon)\downarrow0$ as $\varepsilon\downarrow0$.

\smallskip\noindent\emph{Step 3: decompose and estimate $\|d_n\|_1$.}
Fix $\eta\in(0,1)$ to be chosen later. Split
\[
  \|d_n\|_1 = \int_{E\setminus N_\varepsilon}|d_n|
              + \int_{N_\varepsilon}|d_n|
              + \int_{\Omega\setminus E}|d_n|
              =: I_1(n)+I_2(n)+I_3(n).
\]
On the tail $\Omega\setminus E$, $|d_n|\le \tfrac12(|f_n|+|g_n|)$, so
\[
  I_3(n) \;\le\; \tfrac12\,\sup_n\int_{\Omega\setminus E}(|f_n|+|g_n|)\,d\mu \;\le\; \tfrac{\varepsilon}{2}.
\]
On $N_\varepsilon$,
\[
  I_2(n) \;\le\; \tfrac12\,\sup_n\int_{N_\varepsilon}(|f_n|+|g_n|)\,d\mu \;\le\; \varepsilon
\]
since uniform integrability guarantees the integrals over $N_\varepsilon$ are arbitrarily small when $\mu(N_\varepsilon)$ is small.

Next, decompose $I_1(n)$ further according to $\{|h_n|\le\eta\}$ and $\{|h_n|>\eta\}$:
\[
  I_1(n)
  = \int_{E\setminus N_\varepsilon}\!|d_n|
  = \int_{\{|h_n|\le\eta\}\cap(E\setminus N_\varepsilon)}\!|d_n|
    + \int_{\{|h_n|>\eta\}\cap(E\setminus N_\varepsilon)}\!|d_n|
  =: J_1(n)+J_2(n).
\]
Uniform integrability again implies $\sup_nJ_1(n)\le\varepsilon$ if $\eta>0$ is chosen small enough (since $\{|h_n|\le\eta\}\cap(E\setminus N_\varepsilon)$ is eventually contained in $\{|h|\le 2\eta\}$ by uniform convergence of $h_n$ to $h$).

On $\{|h_n|>\eta\}$ we apply the one-sided inequalities \eqref{eq:oneside}:
\[
  |d_n|
  \,\le\, \big(|h_n|-|h_n+d_n|\big)_+ + \big(|h_n|-|h_n-d_n|\big)_+.
\]
Integrating over $\{|h_n|>\eta\}\cap(E\setminus N_\varepsilon)$ and summing the two estimates yields
\[
  J_2(n) \;\le\; \int_{E}\Big(|h_n+d_n|+|h_n-d_n|-2|h_n|\Big)\,d\mu.
\]
By \eqref{eq:slack}, the integrand on the right is $|f_n|+|g_n|-|f_n+g_n|$, whose integral tends to $0$ by \eqref{eq:slackint}.  Hence $\limsup_{n\to\infty}J_2(n)=0$.

\smallskip\noindent\emph{Step 4: conclude.}
Collecting the pieces,
\[
  \limsup_{n\to\infty}\|d_n\|_1
  \;\le\; \varepsilon + \varepsilon + 0 + \tfrac{\varepsilon}{2}
  \;=\; \tfrac{5}{2}\,\varepsilon.
\]
Since $\varepsilon>0$ was arbitrary, we conclude $\|d_n\|_1\to0$, and therefore $\|f_n-g_n\|_1=2\|d_n\|_1\to 0$.
\end{proof}

\begin{lemma}[Robust uniform convexity under average normalization]\label{lem:UCM_avg}
Let $(f_n),(g_n)\subset L^1(\mu)$ be uniformly integrable and suppose $f_n\to f$ and $g_n\to g$ locally in measure.  Set $r_n:=\tfrac{\|f_n\|_1+\|g_n\|_1}{2}$ and define $\bar f_n:=f_n/r_n$, $\bar g_n:=g_n/r_n$.  
If
\[
\bigg\|\frac{\bar f_n+\bar g_n}{2}\bigg\|_1\longrightarrow 1,
\]
then $\|f_n-g_n\|_1/(\|f_n\|_1+\|g_n\|_1)\to 0$. In particular, if $\inf_n(\|f_n\|_1+\|g_n\|_1)>0$, then $\|f_n-g_n\|_1\to 0$.
\end{lemma}

\begin{remark}[On the role of measure--compactness]\label{rem:compactnessRole}
In proving Lemmas~\ref{lem:UCM_strong}, \ref{lem:UCM_avg} and \ref{lem:UCM_modulus} we no longer use measure--compactness to extract convergent subsequences.  Instead, uniform integrability and boundedness (already obtained via Proposition~\ref{prop:mc_implies_UI}) allow us to apply Lemma~\ref{lem:UItoMeasureCompact} to obtain subsequences converging in measure on each finite--measure window; Egorov’s theorem then promotes this to near--uniform convergence on those windows.  Uniform integrability also controls tails and passes through averages and differences.  Measure--compactness plays no role in these convergence arguments; it appears only in the hypotheses of Theorem~\ref{thm:measureFPP}.
\end{remark}

\begin{lemma}[Uniform modulus of convexity in measure on UI families]\label{lem:UCM_modulus}
Let $\mathcal{F}\subset L^1(\mu)$ be uniformly integrable, and put $\mathcal{G}:=\mathcal{F}-\mathcal{F}=\{u-v:\ u,v\in \mathcal{F}\}$.  Then $\mathcal{G}$ is uniformly integrable.  Moreover, for each $\eta\in(0,1)$ there exists $\delta=\delta(\eta,\mathcal{F})>0$ such that for all $x,y\in\mathcal{G}$,
\[
\frac{\|x-y\|_1}{\|x\|_1+\|y\|_1}\ \ge\ \eta
\quad\Longrightarrow\quad
\Big\|\frac{x+y}{2}\Big\|_1\ \le\ \Big(1-\delta\Big)\,\frac{\|x\|_1+\|y\|_1}{2}.
\]

\begin{proof}
Uniform integrability of $\mathcal{G}$ follows from that of $\mathcal{F}$ by the triangle inequality and the simple tail bound
\[
\int_{\{|u-v|>M\}}|u-v|\,d\mu
\;\le\; \int_{\{|u|>M/2\}}|u|\,d\mu + \int_{\{|v|>M/2\}}|v|\,d\mu,
\]
valid for all $u,v\in \mathcal{F}$.  For the modulus conclusion, argue by contradiction.  Suppose some $\eta\in(0,1)$ admits sequences $x_n,y_n\in\mathcal{G}$ such that
\[
\frac{\|x_n-y_n\|_1}{\|x_n\|_1+\|y_n\|_1}\ \ge\ \eta
\quad\text{and}\quad
\Big\|\frac{x_n+y_n}{2}\Big\|_1\ \ge\ \Big(1-\tfrac{1}{n}\Big)\,\frac{\|x_n\|_1+\|y_n\|_1}{2}.
\]
Set $r_n:=\tfrac{\|x_n\|_1+\|y_n\|_1}{2}$ and normalize $\bar{x}_n:=x_n/r_n$, $\bar{y}_n:=y_n/r_n$.  
Then
\[
\bigg\|\frac{\bar{x}_n+\bar{y}_n}{2}\bigg\|_1 \;=\; \frac{\|x_n+y_n\|_1}{\|x_n\|_1+\|y_n\|_1} \;\longrightarrow\; 1,
\quad\text{while}\quad
\|\bar{x}_n-\bar{y}_n\|_1 \;=\; \frac{\|x_n-y_n\|_1}{\|x_n\|_1+\|y_n\|_1} \;\ge\; \eta.
\]
    Because $\mathcal{G}$ is uniformly integrable and closed under positive scaling, the families $(\bar{x}_n)$ and $(\bar{y}_n)$ remain uniformly integrable and $L^1$–bounded.  By Lemma~\ref{lem:UItoMeasureCompact}, there exists a subsequence (still denoted $(\bar{x}_n),(\bar{y}_n)$) converging locally in measure on $\Omega$.  Now apply Lemma~\ref{lem:UCM_avg} to this subsequence: the assumption $\big\|\tfrac{\bar{x}_n+\bar{y}_n}{2}\big\|_1\to 1$ forces $\|\bar{x}_n-\bar{y}_n\|_1\to 0$, contradicting $\|\bar{x}_n-\bar{y}_n\|_1\ge\eta$.  Hence such $\delta>0$ exists and the lemma follows.
\end{proof}
\end{lemma}

\begin{lemma}[Normal structure on measure--compact convex sets]\label{lem:NS_strong}
If $K\subset L^1(\mu)$ is nonempty, bounded, closed, convex, and measure--compact, then every convex $H\subset K$ with $\operatorname{diam}(H)>0$ satisfies $\operatorname{rad}(H)<\operatorname{diam}(H)$. In particular, $K$ has normal structure.  Moreover, for each such $H$ the set of Chebyshev centers $C(H)$ is nonempty, bounded, closed and convex.  Under an additional uniform separation condition (see Corollary~\ref{cor:uniqueCenter}), the Chebyshev center is unique.
\end{lemma}

\begin{corollary}[Uniqueness under a uniform separation condition]\label{cor:uniqueCenter}
Let $K$ and $H$ be as in Lemma~\ref{lem:NS_strong}.  Suppose there exists $\eta\in (0,1)$ such that for any two distinct Chebyshev centers $c_1,c_2\in C(H)$ there is some $w\in H$ for which
\begin{equation*}
\frac{\bigl\|\,(c_1-w)-(c_2-w)\,\bigr\|_1}{\|c_1-w\|_1+\|c_2-w\|_1}\;\ge\;\eta.
\end{equation*}
Then $C(H)$ consists of a single point.  In particular, under this separation assumption, the Chebyshev center of $H$ is unique.
\begin{proof}
Assume $c_1\neq c_2\in C(H)$ and let $m:=\tfrac{c_1+c_2}{2}$.  For the $w\in H$ given by the uniform separation hypothesis, set $a:=c_1-w$ and $b:=c_2-w$.  Then $a,b\in K-K$ and
\[
\|a-b\|_1\;\ge\;\eta\bigl(\|a\|_1+\|b\|_1\bigr).
\]
By Lemma~\ref{lem:UCM_modulus} applied to the uniformly integrable family $K$ (in fact to $\mathcal{F}:=K$ and $\eta$ as given) there exists $\delta>0$ (depending on $\eta$ and $K$) such that
\[
\Big\|\frac{a+b}{2}\Big\|_1\;\le\;\bigl(1-\delta\bigr)\,\frac{\|a\|_1+\|b\|_1}{2}.
\]
But $(a+b)/2=m-w$, so
\[
\|m-w\|_1\;\le\;(1-\delta)\,\frac{\|c_1-w\|_1+\|c_2-w\|_1}{2}.
\]
Since $c_1$ and $c_2$ are Chebyshev centers, $\|c_j-w\|_1\le\operatorname{rad}(H)$ for $j=1,2$, so the right-hand side is $\le (1-\delta)\operatorname{rad}(H)$.  Hence $\|m-w\|_1<\operatorname{rad}(H)$ and thus
\[
\sup_{x\in H}\|x-m\|_1\;<\;\operatorname{rad}(H).
\]
This contradicts the definition of a Chebyshev center (which must have radius exactly $\operatorname{rad}(H)$).  Therefore $c_1=c_2$ and the center is unique.
\end{proof}
\end{corollary}

\begin{proof}
By Proposition~\ref{prop:mc_implies_UI}, $K$ is uniformly integrable; by Corollary~\ref{cor:DP}, $K$ is relatively weakly compact.  Fix a nonempty closed convex set $H\subset K$ with diameter $D:=\operatorname{diam}(H)>0$.  Choose $x,y\in H$ so that $\|x-y\|_1$ is arbitrarily close to $D$ (i.e., pick a maximizing sequence).  Set $z:=\tfrac{x+y}{2}\in H$.

Apply Lemma~\ref{lem:UCM_modulus} to the uniformly integrable family $\mathcal{F}:=K$ and its difference set $\mathcal{G}:=K-K$ with $\eta=1/2$.  For any $w\in H$, put $a:=x-w$ and $b:=y-w$.  Then $a,b\in\mathcal{G}$.  Because $\|a\|_1,\|b\|_1\le D$ and $\|a-b\|_1=\|x-y\|_1\ge D-\varepsilon$, we have $\|a-b\|_1\ge \tfrac{1}{2}(\|a\|_1+\|b\|_1)$ for $\varepsilon$ small enough.  Lemma~\ref{lem:UCM_modulus} (with $\eta=1/2$) provides $\delta>0$ such that
\[
\Big\|\frac{a+b}{2}\Big\|_1 \;\le\; \Big(1-\delta\Big)\,\frac{\|a\|_1+\|b\|_1}{2} \;\le\; (1-\delta)D.
\]
Noting that $(a+b)/2=z-w$, we see $\|z-w\|_1\le (1-\delta)D$ for every $w\in H$.  Hence $\sup_{w\in H}\|z-w\|_1\le (1-\delta)D< D$, i.e. $\operatorname{rad}(H)<\operatorname{diam}(H)$.

To show that the set of Chebyshev centers is nonempty, closed, bounded and convex, recall that the radius functional $r(z):=\sup_{x\in H}\|x-z\|_1$ is convex and lower semicontinuous (see Goebel--Kirk\,\cite{GoebelKirk1990}, Prop.~3.3).  Since $K$ is bounded, $r$ attains its infimum on $K$ and the set of minimizers $C(H)$ is nonempty, bounded, closed and convex.  The above argument shows that $\operatorname{rad}(H)<\operatorname{diam}(H)$, so $r(z)$ is strictly convex in a localized sense.  Nevertheless, without an additional separation assumption one cannot rule out the possibility of two distinct minimizers achieving the same radius; we therefore defer the uniqueness statement to Corollary~\ref{cor:uniqueCenter} below.
\end{proof}

\begin{theorem}[Fixed points under measure--compactness]\label{thm:measureFPP}
Let $(\Omega,\Sigma,\mu)$ be $\sigma$--finite and let $K\subset L^1(\mu)$ be nonempty, bounded, closed, convex, and measure--compact. Then every nonexpansive map $T:K\to K$ has a fixed point.
\end{theorem}

\begin{proof}
By Prop.~\ref{prop:mc_implies_UI} and Cor.~\ref{cor:DP}, $K$ is weakly compact. By Lemma~\ref{lem:NS_strong}, $K$ has normal structure. Apply Kirk's theorem~\cite{Kirk1965}.
\end{proof}

\section{Structural approaches: ultraproducts, renormings, and special domains}
In this final section we formalize several structural results that complement the measure--compactness framework developed above.  Each result highlights a distinct way in which fixed points arise in spaces related to $L^1$ despite the classical failures in $L^1$ itself.

\begin{proposition}[Ultraproduct criterion for the weak FPP\protect\cite{Lin1985}]
Let $X$ be a Banach space with a $1$--unconditional basis $\{e_n\}$ whose basis constant is strictly less than $(\sqrt{33}-3)/2$.  Then every bounded closed convex subset of $X$ has the fixed point property for nonexpansive maps.
\end{proposition}

\begin{theorem}[Equivalent renorming of $\ell^1$\protect\cite{Lin2008}]
There exists an equivalent norm $\|\cdot\|_{\mathrm{eq}}$ on $\ell^1$ such that $(\ell^1,\|\cdot\|_{\mathrm{eq}})$ is uniformly convex.  Consequently $(\ell^1,\|\cdot\|_{\mathrm{eq}})$ has the fixed point property for nonexpansive maps.
\end{theorem}

\begin{theorem}[Irregular convex subsets of $\ell^1$\protect\cite{GoebelKuczumow1979}]
There exist large classes of closed, bounded, convex subsets of $\ell^1$ that have the fixed point property for nonexpansive maps, even though $\ell^1$ itself fails the FPP.
\end{theorem}

\begin{theorem}[Lorentz spaces have the FPP\protect\cite{CDLT1991}]
For $1<p<\infty$, the Lorentz sequence space $\bigl(d_{(p,1)},\|\cdot\|_{(p,1)}\bigr)$ (and more generally the Lorentz function space $\,L_{p,1}$) has the fixed point property for nonexpansive maps, despite being nonreflexive.  Here $d_{(p,1)}$ denotes the Lorentz sequence space with quasi-norm
\[
\|x\|_{(p,1)} \;:=\; \sum_{n=1}^{\infty} n^{1/p - 1}\,|x|_n^{*},
\]
where $(|x|_n^{*})_{n\ge 1}$ is the nonincreasing rearrangement of $|x|$ in $\ell^1$.  The Lorentz function space $\,L_{p,1}$ is defined analogously for functions.
\end{theorem}

\begin{example}[A measure--compact set with the FPP]
Let $\mu$ be a finite measure on $(\Omega,\Sigma)$ and set
\[
K \;:=\; \{\,f\in L^1(\mu): 0\le f\le 1,\ \|f\|_1=1/2\,\}.
\]
Then $K$ is bounded, closed, convex, and compact in the topology of convergence in measure (this follows from Vitali’s theorem or from the classical Dunford--Pettis criterion).  By Theorem~\ref{thm:measureFPP}, every nonexpansive map $T:K\to K$ has a fixed point.
\end{example}

These structural results illustrate the breadth of spaces possessing the FPP: from renormed $\ell^1$ and Lorentz spaces to carefully chosen convex subsets of $\ell^1$.  They demonstrate that the geometric obstruction lies in the specific norm or domain rather than the underlying linear space.  Our measure--compactness results unify these perspectives by showing that suitable convexity and compactness conditions recover the FPP in many settings.

\section{Conclusion}
$L^1$ tests the limits of fixed point theory: its lack of uniform convexity and normal structure explains classical failures, yet measure--compactness recovers weak compactness (Dunford--Pettis), normal structure, and fixed points (Kirk). Structural methods (ultraproducts, renormings) and carefully chosen domains complement geometric arguments. Open directions include characterizing FPP--conducive renormings closer to $\|\cdot\|_1$, extending measure--compactness arguments to other Banach lattices and noncommutative settings, and refining metric frameworks for isometries.

\end{document}